\documentclass[pdflatex,default]{sn-jnl}

\jyear{2021}%

\usepackage{amsfonts,amsmath,amsthm}
\usepackage[algo2e,linesnumbered,vlined,ruled]{algorithm2e}

\theoremstyle{thmstyleone}%
\newtheorem{theorem}{Theorem}
\newtheorem{proposition}[theorem]{Proposition}%

\theoremstyle{thmstyletwo}%

\theoremstyle{thmstylethree}%
\newtheorem{lemma}[theorem]{Lemma}

\newtheorem{assumptionA}{Assumption}
\newtheorem{assumptionB}{Assumption}

\DeclareMathOperator*{\argmin}{arg\,min}
\newcommand{\bR}{{\mathbb R}}
\newcommand{\cC}{\mathcal{C}}
\newcommand{\cL}{\mathcal L}
\newcommand{\cO}{\mathcal O}
\newcommand{\cH}{\mathcal H}

\begin{document}

\title[Regrets of  PMM  for Online Non-convex Optimization]{Regrets of  Proximal Method of Multipliers  for Online Non-convex Optimization with Long Term Constraints}


\author[1,2]{\fnm{Liwei} \sur{Zhang}}\email{lwzhang@dlut.edu.cn}

\author[1]{\fnm{Haoyang} \sur{Liu}}\email{hyliu@mail.dlut.edu.cn}

\author*[1,2]{\fnm{Xiantao} \sur{Xiao}}\email{xtxiao@dlut.edu.cn}

\affil*[1]{\orgdiv{School of Mathematical Sciences}, \orgname{Dalian University
of Technology}, \orgaddress{ \city{Dalian}, \postcode{116024},  \country{China}}}

\affil[2]{ \orgname{Key Laboratory for Computational Mathematics and Data Intelligence of Liaoning Province}, \orgaddress{ \city{Dalian}, \postcode{116024},  \country{China}}}


\abstract{The online optimization problem with non-convex loss functions over  a closed convex set, coupled with a set of  inequality  (possibly non-convex) constraints is a challenging online learning problem. A   proximal method of multipliers with quadratic approximations  (named as OPMM) is presented to solve this online non-convex  optimization with long term constraints.
 Regrets of the violation of Karush-Kuhn-Tucker conditions of OPMM for solving online non-convex optimization problems are analyzed. Under mild conditions, it is shown  that this algorithm exhibits  ${\cO}(T^{-1/8})$ Lagrangian gradient violation  regret, ${\cO}(T^{-1/8})$ constraint violation regret   and ${\cO}(T^{-1/4})$  complementarity residual regret if parameters in the algorithm are properly chosen,  where $T$ denotes the number of time periods. For the case that the objective is a convex quadratic function, we demonstrate that the regret of the objective reduction can be established even the feasible set is non-convex. For the case when the constraint functions are convex, if  the solution of the subproblem in OPMM is obtained  by solving its dual,  OPMM  is  proved to be an implementable  projection  method for solving the online  non-convex optimization problem.}




\maketitle

\section{Introduction}\label{sec:intro}
In recent years, a number of efficient algorithms have
been developed for online optimization. Convexity of the loss functions and the constraint sets has played a central
role in the development of many of these algorithms. In this paper, we consider a more
general setting, where the sequence of loss functions encountered by the learner could be
non-convex and the constraint set is defined by a set of (possibly non-convex) inequalities.
 Such a setting has various applications in machine learning \cite{MSF2017,NPM2019,CJGWNYS2019}, especially in
adversarial training \cite{SZSEEGF2014} and training of Generative
Adversarial Networks (GANs) \cite{Goodfellow2014}.

Most of the existing works about online optimization have focused
on convex loss functions. A number of computationally efficient approaches
have been proposed for regret minimization in this setting. Among them the famous ones include Follow-the-leader \cite{Kalai2005},  Follow-the-Regularized-Leader  \cite{Shai2007a,Shai2007b},  Exponentiated Online Gradient \cite{Kiv1997}, Online Mirror Descent, Perceptron \cite{Rosenblatt1958} and Winnow \cite{Littlestone1988}. There are also a  lot of publications concerning algorithms for online convex optimization, see \cite[Chapter 7]{MRTal2012}, \cite[Chapter 21]{SSS2014}, and  survey papers  \cite{Shai2011,Hazan2015} and references cited in these two papers.

However, when the loss functions are
non-convex or the constraint sets are non-convex, minimizing the regret is computationally prohibitive. In the last years, there have been several papers about learning with
non-convex losses over simple convex constraint sets. A few heuristic algorithms have been
proposed in \cite{EBG2011,GPSB2011}
without  establishing the regret bounds.
In \cite{GaoZZhang2018},  the regret of online projection gradient method for a restricted
class of loss functions is analyzed.
 The notion of local regret and  the regret of online gradient method for a class of continuously differentiable non-convex loss  functions are presented in \cite{HSZ2017}. DC (difference of convex functions) programming and
DCA method for online learning problems with non-convex loss functions are investigated in \cite{LH2020}. In \cite{YDHTW2018}, a recursive exponential weighted algorithm that attains a regret of ${\cO}(T^{-1/2})$ for non-convex Lipschitz continuous loss functions is proposed.
It is   shown  in \cite{AgarwalHazn2019} that the classical Follow-the-Perturbed-Leader (FTPL) algorithm achieves ${\cO}(T^{-1/3})$ regret for general non-convex losses which are
Lipschitz continuous.  Moreover, in \cite{SN2019}, it is proved that FTPL achieves optimal regret rate ${\cO}(T^{-1/2})$ for the problem of online learning with non-convex losses.   An online
 cubic-regularized Newton method for non-convex online optimization is studied in \cite{RBGM2019}.

 In this paper, we  consider a more complicated non-convex online optimization problem, which has a complex constraint set defined  by
\begin{equation}\label{onlineP}
 \Phi=\{x \in \cC: g_i(x) \leq 0,\ i=1,\ldots, p\}.
\end{equation}
Here,  $\cC \subset \bR^n$  is a nonempty convex compact set with diameter  $D_0:=\sup_{x,x'\in \cC}\|x-x'\|$ and  $g_i: \bR^n \rightarrow \bR$, $i=1,\ldots,p$ are continuous (possibly non-convex) functions.

In order to alleviate the computational challenge of the projection $\Pi_{\Phi}(\cdot)$ with $\Phi$ defined by (\ref{onlineP}), in \cite{MJY2012} the authors considered to relax the constraints $g_i(x)\leq 0$ to be long term  constraints. That is, the decision $x^t$ is not required to satisfy $g_i(x_t)\leq 0$ at each round, but only require that $\sum_{t=1}^Tg_i(x^t)\leq 0$.
There are some recent works related to online convex optimization with long term constraints.  In \cite{MJY2012},
a gradient based algorithm is designed  to achieve ${\cO}(T^{-1/2})$ regret bound and ${\cO}(T^{-1/4})$   violation of constraints for an online optimization problem whose constraint set is defined by a set of inequalities of smooth convex functions.  In
 \cite{JHA2016,YNeely2016}  new algorithms are developed to improve the performance  in comparison with \cite{MJY2012}. However, for non-convex online optimization problems with non-convex loss functions and constraint sets of the form (\ref{onlineP}), the research has been very limited until recently.


  At round $t$, we consider the following non-convex optimization problem
\begin{equation}\label{eq:Pt}
\begin{array}{ll}
\min\limits_{x\in \cC} & f_t(x)\\[4pt]
{\rm s.t.} & g(x) \leq 0,
\end{array}
\end{equation}
where $g(x):=(g_1(x),\ldots,g_p(x))^T$. Since Problem (\ref{eq:Pt}) is non-convex, it is unrealistic to analyze the regrets in both objective reduction and constraint violation. Just like the offline non-convex optimization, it is natural to consider the Karush-Kuhn-Tucker (KKT) conditions which are given by
\begin{equation}\label{KKTpt}
\begin{array}{l}
0\in \nabla_xL^t(x,\lambda)+N_{\cC}(x),\\[10pt]
0 \geq g(x) \perp \lambda\geq 0,
\end{array}
\end{equation}
where $L^t(x,\lambda)=f_t(x)+\lambda^Tg(x)$ and $N_{\cC}(x)$ is the normal cone of $\cC$ at $x$.
Conditions (\ref{KKTpt}) are equivalent to the following equalities:
\begin{equation}\label{KKTge}
\begin{array}{l}
{\rm dist}\, \Big(0, \nabla_xL^t(x,\lambda)+N_{\cC}(x)\Big)=0,\\[10pt]
\lambda-[\lambda+ g(x)]_+=0,
\end{array}
\end{equation}
where $[\cdot]_+:=\max\{0,\cdot\}$.
Therefore, it is reasonable to consider the regret of violation for the equalities in (\ref{KKTge}).

In this paper, we extend the proximal method of multipliers, a classical algorithm proposed in \cite{Rockafellar76A} to solve convex programming, for  online non-convex optimization problem, and analyze its regret bounds  for KKT violation consisting of Lagrangian residual, constraint violation and complementarity residual. Let
 $q^t_i(x)$, $i=0,1,\ldots,p$ be the quadratic approximations of $f_t$ and $g_i,i=1,\ldots,p$ at $x^t$, respectively, defined by
\[
\begin{array}{l}
q^t_0(x):= f_t(x^t)+\langle \nabla f_t(x^t), x-x^t \rangle+ \frac{1}{2}\left \langle \Theta^t_0(x-x^t), x-x^t\right\rangle,\\[10pt]
 q^t_i(x):=g_i(x^t)+\langle  \nabla g_i(x^t), x-x^t \rangle+ \frac{1}{2} \left\langle \Theta^t_i(x-x^t), x-x^t\right\rangle,\ i=1,\ldots, p,
 \end{array}
\]
where  $\Theta^t_0 \in \mathbb S^n$ and  $\Theta^t_i \in \mathbb S^n$ are properly selected symmetric $n\times n$ matrices.
The corresponding augmented Lagrangian function  is defined by
\begin{equation}\label{augL}
\cL^t_{\sigma}(x,\lambda): =q_0^t(x) + \frac{1}{2\sigma}\left[ \sum_{i=1}^p[\lambda_i+\sigma q^t_i(x)]_+^2-\|\lambda\|^2\right]
\end{equation}
for $(x,\lambda)\in \bR^n \times \bR^p$ and $\sigma>0$. At each round $t$, we let $x^{t+1}$ be the optimal solution of the following problem
\[
\min_{x \in \cC} \, \cL^t_{\sigma }(x,\lambda^t) +\frac{\alpha}{2}\|x-x^t\|^2,
\]
and update the multipliers by $\lambda^{t+1}_i=[\lambda^t_i+\sigma q_i^t(x^{t+1})]_+,\ i=1,\ldots,p$,
where $\alpha>0$ is some parameter. Let
$q^t(x):=(q_1^t(x),\ldots,q_p^t(x))^T$,
then
$\lambda^{t+1}=[\lambda^t+\sigma q^t(x^{t+1})]_+$.
In detail, the online  proximal method of multipliers (OPMM) with quadratic approximations for the non-convex online optimization problem with constraint set  (\ref{onlineP}) can be described
in Algorithm \ref{algo:OPMM}.

\begin{algorithm2e}[!htp]\caption{
An online proximal method of multipliers (OPMM) with quadratic approximations.}\label{algo:OPMM}
  \SetKwInOut{Input}{Input}
  \Input{$\lambda^1=0$, $x^1 \in \cC$, $\sigma >0$ and $\alpha>0$, receive a cost function $f_1(\cdot)$. }
  \BlankLine
  \For{$t\leftarrow 1$ \KwTo $T$}{
  \emph{Choose $\Theta^t_0 \in \mathbb S^n$ and  $\Theta^t_i \in \mathbb S^n$, $i=1,\ldots, p$ such that $q^t_0(\cdot)$ and $q^t_i(\cdot)$  are proper quadratic approximations of $f_t(\cdot)$ and $g_i(\cdot)$ at $x^t$, respectively.}\

  \emph{Compute}\
 \begin{equation}\label{xna}
x^{t+1}= \argmin_{x \in \cC} \,\left\{ \cL^t_{\sigma }(x,\lambda^t) +\frac{\alpha}{2}\|x-x^t\|^2\right\}.
\end{equation}	

  \emph{Update}\
\[
\lambda^{t+1}_i=[\lambda^t_i+\sigma q_i^t(x^{t+1})]_+,\ i=1,\ldots,p.
\]

  \emph{Receive a cost function $f_{t+1}(\cdot)$.}\
  }
\end{algorithm2e}


The main results of this paper can be summarized as follows.
\begin{itemize}
\item When we choose $\sigma=T^{-1/4}$ and $\alpha=T^{1/4}$,  under mild assumptions, there exists $w^{t+1} \in N_{\cC}(x^{t+1})$ for any $t=1,\ldots, T$ such that the regret of Lagrangian residual is
\[
\left\|  \frac{1}{T}\sum_{t=1}^T \left[\nabla f_{t+1}(x^{t+1})+ \sum_{i=1}^p \lambda^{t+1}_i \nabla g_i(x^{t+1})+w^{t+1} \right]\right\|\\[10pt]
  \leq \cO(T^{-1/8}),
\]
the regret of constraint violation is
$$
 \frac{1}{T}\sum_{t=1}^T g_i(x^t)\leq \cO( T^{-1/8}),\quad i=1,\ldots,p,
$$
and the regret of complementarity residual is
$$
 \frac{1}{T}\sum_{t=1}^T\|\lambda^{t+1}-[\lambda^{t+1}+\sigma g(x^{t+1})]_+\| \leq \cO(T^{-1/4}).
$$
\item For the case that the objective function $f_t$ is convex quadratic, if $\sigma=T^{-1/2}$ and $\alpha=T^{1/2}$, the regret of objective reduction is
\[
 \frac{1}{T}\sum_{t=1}^Tf_t(x^t)-\inf_{z \in \Phi}  \frac{1}{T}\sum_{t=1}^Tf_t(z) \leq\cO( T^{-1/2}).
\]
\item For the case that $g_1,\ldots, g_p$ are convex functions, if  the solution of the subproblem in OPMM is obtained  by solving the dual of the subproblem,  OPMM  can be reformulated as  an implementable  projection  method.
\end{itemize}

The remaining parts of this paper are organized as follows. In Section \ref{sec:prop}, we develop properties of OPMM, which play a key role in the regret analysis of OPMM. In Section \ref{sec:regret}, we establish regret bounds of Lagrangian residual, constraint violation and complementarity residual of OPMM for Problem (\ref{eq:Pt}). In Section \ref{sec:cvx-cons}, for the convex constraint set, OPMM is explained as an  implementable  projection  method for solving the online  optimization problem with long term constraints.  We draw a conclusion in Section \ref{sec:conclusion}.

\section{Auxiliary Properties of OPMM}\label{sec:prop}
In this section, we focus on establishing a variety of auxiliary properties of OPMM under some reasonable  assumptions. We begin by introducing two classes of assumptions, in which the first class is about the structure of Problem (\ref{eq:Pt}) and the second class is to ensure that the quadratic approximations $q_i^t(x)$, $i=0,1,\ldots,p$ are well-defined.
\begin{assumptionA}\label{assu:lip}
There exist constants $\kappa_f>0$, $\kappa_g>0$ and $\nu_g>0$ such that for all $x, x' \in \cC$ and $i=1,\ldots, p$,
\[
\lvert f_t(x)-f_t(x')\rvert\leq \kappa_f\|x-x'\|,\  \lvert g_i(x)-g_i(x')\rvert\leq \kappa_g\|x-x'\|,
\]
and $\|g(x)\|\leq \nu_g$.
\end{assumptionA}

\begin{assumptionA}\label{assu:smooth}
The functions $f_t$ and $g_i$, $i=1,\ldots,p$ are continuously differentiable over  $\cC$.
There exist constants $L_f>0$ and $L_g>0$  such that for all $x, x' \in \cC$ and $i=1,\ldots, p$,
\[
\|\nabla f_t(x)-\nabla f_t(x')\|\leq L_f\|x-x'\|,\  \|\nabla g_i(x)-\nabla g_i(x')\|\leq L_g\|x-x'\|.
\]
\end{assumptionA}

\begin{assumptionA}\label{assu:slater}
The Slater condition holds, that is,
there exist a constant $\varepsilon_0>0$ and a vector $\widehat x \in \cC$ such that
\[
     g_i(\widehat x) \leq -\varepsilon_0, \, i=1,\ldots, p.
\]
\end{assumptionA}

Note that the set $\cC$ is bounded, if Assumption \ref{assu:smooth} holds true, we have  that Assumption \ref{assu:lip} is satisfied. Indeed, if Assumption \ref{assu:smooth} holds, it follows that $g(x)$, $\nabla f_t(x)$ and $\nabla g_i(x)$ are bounded over $\cC$, and hence $f_t(x)$ and $g_i(x)$ are Lipschitz continuous. From Assumption \ref{assu:slater} and $\|g(x)\|\leq\nu_g$ in Assumption \ref{assu:lip} it follows that $\nu_g\geq\|g(\widehat x)\|\geq \sqrt{p}\varepsilon_0$, which implicitly implies that $\nu_g\geq \varepsilon_0$.

\begin{assumptionB}\label{assu:psd}
The matrix $\Theta^t_0$ is positively semidefinite.
\end{assumptionB}

\begin{assumptionB}\label{assu:q}
It holds that $q^t_i(x)\leq g_i(x)$, $i=1,\ldots,p$ for all $x\in\cC$.
\end{assumptionB}

\begin{assumptionB}\label{assu:matrix}
There exists a  constant $\kappa_{q}>0$ such that  $\|\Theta^t_i\|\leq \kappa_{q}$ for $i=0,1,\ldots, p$.
\end{assumptionB}

\begin{assumptionB}\label{assu:convex}
The augmented Lagrangian function
$
\cL^t_{\sigma }(\cdot,\lambda^t)
$ is  convex.
\end{assumptionB}

 Roughly speaking, the role of Assumptions \ref{assu:psd}--\ref{assu:convex} is to let the functions $q_i^t$ be conservatively convex approximations to $f_t$ and $g_i$, $i=1,\ldots,p$, respectively, and let the subproblem (\ref{xna}) in OPMM be easily solvable.
We remark that Assumption \ref{assu:convex} is satisfied if all matrices $\Theta_i^t$, $i=0,1,\ldots,p$ are positively semidefinite.

\begin{lemma}\label{lem:2}
Let Assumptions  \ref{assu:lip}, \ref{assu:matrix} be satisfied.   Then, for  $i=1,\ldots,p$,
\[
 \sum_{t=1}^Tg_i(x^t)\leq  \frac{1}{\sigma}\lambda^{T+1}_i+\gamma\kappa_g^2T +\left[ \frac{1}{4\gamma} + \frac{\kappa_{q}}{2}\right]\sum_{t=1}^T\|x^{t+1}-x^t\|^2,
\]
where $\gamma>0$ is an arbitrary  scalar.
\end{lemma}
\begin{proof} From the relation $\lambda^{t+1}_i=[\lambda^t_i+\sigma q^t_i(x^{t+1})]_+$ and the fact that $[a]_+\geq a$ for any scalar $a$, we have
$$
\begin{array}{ll}
\lambda^{t+1}_i
&\geq \lambda^t_i+\sigma\left( g_i(x^t)+\langle \nabla g_i(x^t), x^{t+1}-x^t\rangle + \frac{1}{2} \left \langle\Theta^t_i (x^{t+1}-x^t),x^{t+1}-x^t \right \rangle \right)\\[10pt]
& \geq \lambda^t_i+\sigma\left(g_i(x^t)- \|\nabla g_i(x^t)\|\|x^{t+1}-x^t\|- \frac{1}{2} \|\Theta^t_i\|\| x^{t+1}-x^t\|^2\right),
\end{array}
$$
which, together with Assumptions  \ref{assu:lip}, \ref{assu:matrix}, implies for any $\gamma>0$ that
\[
\frac{1}{\sigma}(\lambda^{t+1}_i- \lambda^t_i)\geq g_i(x^t)-\gamma\kappa_g^2-\left(\frac{1}{4\gamma}+\frac{\kappa_q}{2}\right)\|x^{t+1}-x^t\|^2.
\]
Summing up the above inequality from $t=1$ to $T$, rearranging terms and noticing that $\lambda^1=0$, we derive the claim.
\end{proof}

In order to obtain a bound of $ \sum_{t=1}^T g_i(x^t)$ in Lemma \ref{lem:2}, we need to estimate an upper bound of $\sum_{t=1}^T\|x^{t+1}-x^t\|^2$, which is given in the following lemma.
\begin{lemma}\label{lem:3}
Let Assumptions  \ref{assu:lip}, \ref{assu:psd}, \ref{assu:q} be satisfied.  Then, for  any $\alpha > 0$,
\[
 \sum_{t=1}^T\|x^{t+1}-x^t\|^2 \leq  \frac{4}{\alpha}\left[  \frac{T}{\alpha}\kappa_f^2+\nu_g \sum_{t=1}^T\|\lambda^t\|+\frac{\sigma}{2}\nu_g^2T\right].
\]
\end{lemma}
\begin{proof} In view of  (\ref{xna}), it follows from Assumption \ref{assu:q} that
$$
\begin{array}{l}
\langle \nabla f_t(x^t), x^{t+1}-x^t\rangle + \frac{1}{2} \left \langle \Theta^t_0(x^{t+1}-x^t),x^{t+1}-x^t\right \rangle+ \frac{1}{2\sigma}\|\lambda^{t+1}\|^2 + \frac{\alpha}{2}\|x^{t+1}-x^t\|^2\\[10pt]
 \leq  \frac{1}{2\sigma}\sum_{i=1}^p[\lambda_i^t+\sigma q^t_i(x^t)]_+^2
 \leq  \frac{1}{2\sigma}\sum_{i=1}^p[\lambda_i^t+\sigma g_i(x^t)]_+^2
\leq  \frac{1}{2\sigma}\|\lambda^{t} +\sigma g(x^t)\|^2,
\end{array}
$$
which, together with Assumption  \ref{assu:lip} and \ref{assu:psd}, implies  that
$$
\begin{array}{ll}
 \frac{\alpha}{4}\|x^{t+1}-x^t\|^2\\[4pt]
  \leq  \langle \nabla f_t(x^t),x^t-x^{t+1}\rangle-  \frac{\alpha}{4}\|x^{t+1}-x^t\|^2- \frac{1}{2} \left \langle \Theta^t_0(x^{t+1}-x^t),x^{t+1}-x^t\right \rangle\\[9pt]
 \quad\quad + \frac{1}{2\sigma}[\|\lambda^t\|^2-\|\lambda^{t+1}\|^2]+\langle \lambda^t,g(x^t) \rangle+ \frac{\sigma}{2}\|g(x^t)\|^2\\[10pt]
 \leq \left ( \langle \nabla f_t(x^t),x^t-x^{t+1}\rangle-  \frac{\alpha}{4}\|x^{t+1}-x^t\|^2\right)+ \frac{1}{2\sigma}[\|\lambda^t\|^2-\|\lambda^{t+1}\|^2]+\nu_g \|\lambda^t\|+ \frac{\sigma}{2}\nu_g^2\\[10pt]
\leq  \frac{1}{\alpha}\kappa_f^2+ \frac{1}{2\sigma}[\|\lambda^t\|^2-\|\lambda^{t+1}\|^2]+\nu_g \|\lambda^t\|+ \frac{\sigma}{2}\nu_g^2.
\end{array}
$$
The claim is obtained by summing up the above inequality from $t=1$ to $T$, rearranging terms and noticing that $\lambda^1=0$.
\end{proof}

Combining Lemma \ref{lem:2} and Lemma \ref{lem:3}, we obtain the following result which plays an important role in estimating the constraint violation regret.
\begin{proposition}\label{prop:cregret}
Let Assumptions  \ref{assu:lip}, \ref{assu:psd}, \ref{assu:q}, \ref{assu:matrix} be satisfied.  Then,  for any scalar $\gamma>0$, the following results hold:
\[
 \sum_{t=1}^T g_i(x^t)\leq   \frac{1}{\sigma}\lambda^{T+1}_i+ \gamma\kappa_g^2T+\frac{1}{\alpha}\left[  \frac{1}{\gamma}+2\kappa_{q}\right] \left[  \frac{\kappa_f^2}{\alpha}T+\nu_g \sum_{t=1}^T\|\lambda^t\|+\frac{\sigma}{2}\nu_g^2T\right].
\]
\end{proposition}

We next focus our attention on examining the bound of Lagrangian multiplier $\lambda^t$.
\begin{lemma}\label{lem:1}
Let    Assumptions  \ref{assu:lip}, \ref{assu:matrix} be satisfied. Then
\[
\|\lambda^t\|- \sigma\beta_0 \leq \|\lambda^{t+1}\| \leq \|\lambda^t\|+ \sigma\beta_0,
\]
where
\begin{equation}\label{eq:beta0}
\beta_0:=\left[\nu_g+\sqrt{p} \left(\kappa_gD_0+ \frac{1}{2}\kappa_{q}D_0^2\right)\right].
\end{equation}
\end{lemma}
\begin{proof}
It follows from the nonexpansion property of  $[\cdot]_+$,   Assumption  \ref{assu:lip} and Assumption  \ref{assu:matrix} that
$$
\begin{array}{ll}
\|\lambda^{t+1}-\lambda^t\| & =\|[\lambda^t+\sigma q^t(x^{t+1})]_+-[\lambda^t]_+\|\leq\sigma\|q^t(x^{t+1})\|\\[8pt]
& \leq
\sigma \|g(x^{t})\| + \sigma \|q^t(x^{t+1})-g(x^t)\|\\[8pt]
& \leq \sigma\nu_g+\sigma \left ( \sum_{i=1}^p \left(\|\nabla g_i(x^t)\|\|x^{t+1}-x^t\|+ \frac{1}{2}\|\Theta^t_i\|\|x^{t+1}-x^t\|^2\right)^2\right)^{1/2}\\[8pt]
& \leq \sigma\nu_g+\sigma \left ( \sum_{i=1}^p \left(\kappa_gD_0+ \frac{1}{2}\kappa_{q}D_0^2\right)^2\right)^{1/2}\\[8pt]
& \leq \sigma\left[\nu_g+\sqrt{p} \left(\kappa_gD_0+ \frac{1}{2}\kappa_{q}D_0^2\right)\right],
\end{array}
$$
which completes the proof.
\end{proof}
\begin{lemma}\label{lem:l5}
Let  Assumptions \ref{assu:lip}, \ref{assu:slater}, \ref{assu:psd}, \ref{assu:q}, \ref{assu:matrix}, \ref{assu:convex} be satisfied.
 Let $s > 0$ be an arbitrary integer and
 \begin{equation}\label{eq:theta9}
 \vartheta (\sigma,\alpha,s):= \frac{\varepsilon_0\sigma s}{2}+\beta_0\sigma(s-1)+ \frac{\alpha D_0^2}{\varepsilon_0s}+ \frac{\left(2\kappa_fD_0+\kappa_{q}D_0^2\right)}{\varepsilon_0}+
  \frac{\sigma \nu_g^2}{\varepsilon_0}.
 \end{equation}
Then, it follows that
\begin{equation}\label{eq:6}
\lvert\|\lambda^{t+1}\|-\|\lambda^t\|\rvert\leq \sigma \beta_0,
\end{equation}
where $\beta_0$ is defined by (\ref{eq:beta0}).
Moreover, if $\|\lambda^t\| \geq  \vartheta (\sigma,\alpha,s)$, it holds that
\begin{equation}\label{eq:7}
 \|\lambda^{t+s}\|-\|\lambda^t\|
\leq 
-s  \frac{\sigma \varepsilon_0}{2}.
\end{equation}
\end{lemma}
\begin{proof} Inequality (\ref{eq:6}) follows directly from Lemma \ref{lem:1}.
It remains to prove
$$  \|\lambda^{t+s}\|-\|\lambda^t\|
\leq -s  \frac{\sigma \varepsilon_0}{2}
$$
under the condition that $\|\lambda^t\| \geq  \vartheta (\sigma,\alpha,s)$.

In the sequel, for given positive integer $s$, we suppose that $\|\lambda^t\| \geq  \vartheta (\sigma,\alpha,s)$. For any $l \in \{t,t+1,\ldots, t+s-1\}$, under Assumption \ref{assu:convex}, it follows from (\ref{xna}), i.e.,
\[x^{l+1}= \argmin_{x \in \cC} \,\left\{ \cL^l_{\sigma }(x,\lambda^l) +\frac{\alpha}{2}\|x-x^l\|^2\right\},\] and its optimality conditions that $x^{l+1}$ is also a minimizer of $\cL^l_{\sigma }(x,\lambda^l) +\frac{\alpha}{2}[\|x-x^l\|^2-\|x-x^{l+1}\|^2]$ over $\cC$. Therefore,
$$
\begin{array}{l}
\langle \nabla f_l (x^l),x^{l+1}-x^l \rangle+ \frac{1}{2}\left \langle \Theta^l_0(x^{l+1}-x^l), x^{l+1}-x^l \right \rangle+ \frac{1}{2\sigma}\|\lambda^{l+1}\|^2+ \frac{\alpha}{2}\|x^{l+1}-x^l\|^2\\[10pt]
\leq \langle \nabla f_l (x^l),\widehat x-x^l \rangle+ \frac{1}{2}\left \langle \Theta^l_0(\widehat x-x^l), \widehat x-x^l \right \rangle+ \frac{1}{2\sigma}\|[\lambda^{l}+\sigma q^l(\widehat x)]_{+}\|^2\\[10pt]
\quad \quad+ \frac{\alpha}{2}\left[\|\widehat x-x^l\|^2-\|\widehat x-x^{l+1}\|^2\right]\\[10pt]
\leq \langle \nabla f_l (x^l),\widehat x-x^l \rangle+ \frac{1}{2}\left \langle \Theta^l_0(\widehat x-x^l), \widehat x-x^l \right \rangle+ \frac{1}{2\sigma}\|[\lambda^{l}+\sigma g(\widehat x)]_{+}\|^2\\[10pt]
\quad \quad + \frac{\alpha}{2}\left[\|\widehat x-x^l\|^2-\|\widehat x-x^{l+1}\|^2\right]\\[10pt]
\leq \langle \nabla f_l (x^l),\widehat x-x^l \rangle+ \frac{1}{2}\left \langle \Theta^l_0(\widehat x-x^l), \widehat x-x^l \right \rangle+ \frac{1}{2\sigma}\|\lambda^l\|^2+ \langle \lambda^l,g(\widehat x)\rangle
+\frac{\sigma}{2}\|g(\widehat x)\|^2\\[10pt]
\quad \quad + \frac{\alpha}{2}\left[\|\widehat x-x^l\|^2-\|\widehat x-x^{l+1}\|^2\right],
\end{array}
$$
in which  $\widehat{x}$ is given in Assumption \ref{assu:slater} and the second inequality above is obtained  from Assumption \ref{assu:q}.
Reorganizing terms and using   Assumptions \ref{assu:lip}, \ref{assu:psd}, \ref{assu:matrix},
we obtain
\begin{equation}\label{eq:8}
\begin{array}{ll}
 \frac{1}{2\sigma} \left[\|\lambda^{l+1}\|^2-\|\lambda^l\|^2\right]\\[10pt]
\leq \langle \nabla f_l (x^l),\widehat x-x^{l+1} \rangle + \frac{1}{2}\left \langle \Theta^l_0(\widehat x-x^l), \widehat x-x^l \right \rangle+ \langle \lambda^l,g(\widehat x)\rangle
+ \frac{\sigma}{2}\|g(\widehat x)\|^2\\[10pt]
 \quad \quad + \frac{\alpha}{2}\left[\|\widehat x-x^l\|^2-\|\widehat x-x^{l+1}\|^2\right]\\[10pt]
\leq \kappa_fD_0+ \frac{1}{2}\kappa_{q}D_0^2+ \langle \lambda^l,g(\widehat x)\rangle
+ \frac{\sigma}{2}\nu_g^2+ \frac{\alpha}{2}\left[\|\widehat x-x^l\|^2-\|\widehat x-x^{l+1}\|^2\right].
\end{array}
\end{equation}
Noting that for $l \in \{t,t+1,\ldots, t+s-1\}$, one has from Assumption A3 that
\begin{equation}\label{eq:linq}
\langle \lambda^l, g(\widehat x)\rangle= \sum_{j=1}^p\lambda^l_jg_j(\widehat x)
\leq -\varepsilon_0  \sum_{j=1}^p\lambda^l_j
\leq -\varepsilon_0\|\lambda^l\|.
\end{equation}
Thus, making  a summation of (\ref{eq:8}) over $\{t,t+1,\ldots, t+s-1\}$,  noticing
(\ref{eq:linq}) and the fact that $\|\lambda^{t+l}\|\geq \|\lambda^t\|-\sigma \beta_0l$, we obtain
\[
\begin{array}{l}
 \frac{1}{2\sigma} \left[\|\lambda^{t+s}\|^2-\|\lambda^t\|^2\right]\\[10pt]
\leq \left(\kappa_fD_0+ \frac{1}{2}\kappa_{q}D_0^2\right) s +  \frac{\sigma}{2}\nu_g^2s + \sum_{l=t}^{t+s-1}\langle \lambda^l,g(\widehat x)\rangle
 + \frac{\alpha}{2}\left[\left(\|\widehat x-x^t\|^2-\|\widehat x-x^{t+s}\|^2\right) \right]\\[10pt]
\leq \left(\kappa_fD_0+ \frac{1}{2}\kappa_{q}D_0^2\right) s +  \frac{\sigma}{2}\nu_g^2s -\varepsilon_0 \sum_{l=0}^{s-1} \|\lambda^{t+l}\|
 + \frac{\alpha}{2}D_0^2\\[10pt]
\leq \left(\kappa_fD_0+ \frac{1}{2}\kappa_{q}D_0^2\right) s +  \frac{\sigma}{2}\nu_g^2s + \frac{\alpha}{2}D_0^2 -\varepsilon_0 \sum_{l=0}^{s-1} \left[\|\lambda^{t}\|-\sigma\beta_0l \right]
\\[10pt]
\leq \left(\kappa_fD_0+ \frac{1}{2}\kappa_{q}D_0^2\right) s +  \frac{\sigma}{2}\nu_g^2s + \frac{\alpha}{2}D_0^2+ \varepsilon_0\sigma \beta_0  \frac{s(s-1)}{2}
-\varepsilon_0 s\|\lambda^{t}\|,
\end{array}
\]
which, together with $\|\lambda^t\| \geq  \vartheta (\sigma,\alpha,s)$, further implies that
\[
\begin{array}{l}
\|\lambda^{t+s}\|^2\\[10pt]\leq
\|\lambda^t\|^2+2 \sigma \left(\kappa_fD_0+ \frac{1}{2}\kappa_{q}D_0^2\right) s+\sigma^2\nu_g^2s+\alpha \sigma D_0^2
+\varepsilon_0\sigma^2\beta_0s(s-1)-2\varepsilon_0\sigma s \|\lambda^{t}\| \\[10pt]
=\left(\|\lambda^t\|- \frac{\varepsilon_0\sigma}{2}s\right )^2
- \frac{3\varepsilon_0^2\sigma^2}{4}s^2+\varepsilon_0\sigma s \vartheta (\sigma,\alpha,s)-\varepsilon_0\sigma s \|\lambda^{t}\| \\[10pt]
\leq \left(\|\lambda^t\|- \frac{\varepsilon_0\sigma}{2}s\right )^2.
\end{array}
\]
Noticing that $\|\lambda^t\| \geq  \vartheta (\sigma,\alpha,s)\geq\frac{\varepsilon_0\sigma}{2}s$, we have
$\|\lambda^{t+s}\|\leq
\|\lambda^t\|- \frac{\varepsilon_0\sigma}{2}s.
$
The proof is completed.
\end{proof}

The following lemma is a simple variation of  \cite[Lemma 5]{YNeely2017}, which shall be used to deal with KKT violation regret of OPMM. The proof is provided in Appendix \ref{app:lem}.
\begin{lemma}\label{lem:l7}
Let $\{Z_t\}$ be a sequence  with $Z_0 = 0$. Suppose there exist an integer $t_0 >0$, real constants $\theta>0$, $\delta_{\max}>0$ and $ 0 <\zeta \leq \delta_{\max}$ such that
$\lvert Z_{t+1}-Z_t\rvert  \leq \delta_{\max}$
and
\begin{equation}\label{eq:Y9}
Z_{t+t_0}-Z_t  \leq 
-t_0\zeta, \quad \mbox{if } Z_t \geq \theta
\end{equation}
hold for all $t \in \{1,2,\ldots\}.$ Then,
\begin{equation}\label{eq:re-ineq}
Z_t \leq \theta +t_0 \delta_{\max}+t_0  \frac{4 \delta_{\max}^2}{\zeta}\log \left[  \frac{8 \delta_{\max}^2}{\zeta^2} \right], \forall t \in \{1,2,\ldots\}.
\end{equation}
\end{lemma}

If we take $\theta=\vartheta (\sigma,\alpha,s)$, $\delta_{\max}=\sigma \beta_0$, $\zeta =  \frac{\sigma}{2}\varepsilon_0$ and $t_0=s$,
we can observe from $\beta_0\geq\nu_g\geq\varepsilon_0$ and Lemma \ref{lem:l5} that  the conditions in Lemma \ref{lem:l7} are satisfied in terms of $\|\lambda^t\|$.
For convenience, let us introduce
\[
\psi(\sigma,\alpha,s):=
\vartheta (\sigma,\alpha,s)+  \left[\beta_0+  \frac{8\beta_0^2}{\varepsilon_0}\log  \frac{32\beta_0^2}{\varepsilon_0^2}\right] \sigma s.
\]
We can verify that the right-hand side of (\ref{eq:re-ineq})  equals exactly to $\psi(\sigma,\alpha,s)$, that is,
\[
 \theta +t_0 \delta_{\max}+t_0   \frac{4 \delta_{\max}^2}{\zeta}\log \left[   \frac{8 \delta_{\max}^2}{\zeta^2} \right]=\psi(\sigma,\alpha,s).
\]
Therefore, from Lemma \ref{lem:l5} and Lemma \ref{lem:l7} we directly derive the following useful result.
\begin{proposition}\label{prop:lambda}
Let  Assumptions \ref{assu:lip}, \ref{assu:slater}, \ref{assu:psd}, \ref{assu:q}, \ref{assu:matrix}, \ref{assu:convex} be satisfied. Then, for any arbitrary integer $s > 0$,
 the following inequality holds
\[
\|\lambda^t\|\leq \psi(\sigma,\alpha,s).
\]
\end{proposition}

Finally, if we define
\begin{equation}\label{eq:notations}
\begin{array}{ll}
\kappa_0= \frac{\left(2\kappa_fD_0+\kappa_{q}D_0^2\right)}{\varepsilon_0},\,\,
\kappa_1= \frac{D_0^2}{\varepsilon_0},\,\,
\kappa_2= \frac{ \nu_g^2}{\varepsilon_0}-\beta_0,\\[10pt]
\kappa_3=\left[2\beta_0+ \frac{\varepsilon_0}{2}+ \frac{8\beta_0^2}{\varepsilon_0}\log  \frac{32\beta_0^2}{\varepsilon_0^2}\right],
\end{array}
\end{equation}
then $\psi(\sigma,\alpha,s)$ can be rewritten as
\[
\psi(\sigma,\alpha,s)=\kappa_0+\kappa_1 \frac{\alpha}{s}+\kappa_2
 \sigma+\kappa_3 \sigma s.
\]

\section{Regret Analysis of OPMM}\label{sec:regret}
In this section, we establish the regret bounds of the proposed algorithm. In particular, we focus on estimating the following three regrets: regret of Lagrangian residual,  regret of constraint violation and regret of complementarity residual.
The following proposition establishes an upper bound of the so-called Lagrangian residual.
\begin{proposition}\label{prop:kkt}
Let Assumptions \ref{assu:lip}, \ref{assu:smooth}, \ref{assu:matrix} be satisfied. Then, there exists a vector $w^{t+1} \in N_{\cC}(x^{t+1})$ such that
\begin{equation}\label{eq:kkt}
 \left\| \sum_{t=1}^T\cH_t\right\|
\leq
2 \kappa_f+ \frac{\kappa_{q}^2}{2\beta}T+ \frac{(1+p)\beta}{2}
\sum_{t=1}^T\|x^{t+1}-x^t\|^2+ \frac{(L_g+\kappa_{q})^2}{2\beta}
\sum_{t=1}^T\|\lambda^{t+1}\|^2+\alpha D_0,
\end{equation}
where  $\beta>0$ is an arbitrary scalar and
\[
\cH_t:=\nabla f_{t+1}(x^{t+1})+ \sum_{i=1}^p \lambda^{t+1}_i \nabla g_i(x^{t+1})+w^{t+1}.
\]
\end{proposition}
\begin{proof} It follows from the optimality conditions of (\ref{xna}) that
$$
0 \in \nabla q^t_0(x^{t+1})+{\cal J}q^t(x^{t+1})^T\lambda^{t+1}+\alpha(x^{t+1}-x^t)+N_{\cC}(x^{t+1}).
$$
Equivalently, from the definitions of $q_0^t$ and $q^t$, there exists $w^{t+1} \in N_{\cC}(x^{t+1})$ such that
\begin{equation}\label{hel1}
0=\nabla f_t(x^{t})+\Theta_0^t(x^{t+1}-x^t)+\sum_{i=1}^p\lambda_i^{t+1}[\nabla g_i(x^t)+\Theta_i^t(x^{t+1}-x^t)]+\alpha(x^{t+1}-x^t)+w^{t+1}.
\end{equation}
In view of definitions of $q^t_i$, $i=0,1,\ldots,p$ and $\cH_t$, we may rewrite (\ref{hel1}) as
$$
\begin{array}{ll}
0= & \cH_t +[\nabla f_t(x^t)-\nabla f_{t+1}(x^{t+1})]+ \sum_{i=1}^p \lambda^{t+1}_i (\nabla g_i(x^t)-\nabla g_i(x^{t+1}))\\[10pt]
&\quad + \Theta^t_0(x^{t+1}-x^t)+ \sum_{i=1}^p \lambda^{t+1}_i\Theta^t_i(x^{t+1}-x^t)+\alpha(x^{t+1}-x^t).
\end{array}
$$
Making a summation, we obtain
$$
\begin{array}{ll}
&-\sum_{t=1}^T \cH_t \\[10pt] &=  [\nabla f_1(x^1)-\nabla f_{T+1}(x^{T+1})]+ \sum_{t=1}^T\left[ \sum_{i=1}^p \lambda^{t+1}_i (\nabla g_i(x^t)-\nabla g_i(x^{t+1}))\right]\\[10pt]
&\quad +  \sum_{t=1}^T\Theta^t_0(x^{t+1}-x^t)+ \sum_{t=1}^T\left[ \sum_{i=1}^p \lambda^{t+1}_i\Theta^t_i(x^{t+1}-x^t)\right]+\alpha(x^{T+1}-x^1).
\end{array}
$$
Therefore, it follows from Assumptions \ref{assu:lip}, \ref{assu:smooth}, \ref{assu:matrix} that
$$
\begin{array}{l}
 \left\| \sum_{t=1}^T \cH_t\right\|\\[10pt]
\leq [\|\nabla f_1(x^1)\|+\|\nabla f_{T+1}(x^{T+1})\|]+ \sum_{t=1}^T\left[ \sum_{i=1}^p \lambda^{t+1}_i\|\nabla g_i(x^t)-\nabla g_i(x^{t+1})\|\right]\\[10pt]
\quad +  \sum_{t=1}^T\|\Theta^t_0\|\|x^{t+1}-x^t\|+ \sum_{t=1}^T\left[ \sum_{i=1}^p \lambda^{t+1}_i\|\Theta^t_i\|\|x^{t+1}-x^t\|\right]+\alpha D_0\\[10pt]
\leq 2\kappa_f + \kappa_{q}  \sum_{t=1}^T\|x^{t+1}-x^t\|+( L_g+\kappa_{q})
 \sum_{t=1}^T\left[ \sum_{i=1}^p \lambda^{t+1}_i\|x^{t+1}-x^t\|\right]+\alpha D_0\\[10pt]
\leq 2 \kappa_f+ \frac{\kappa_{q}^2}{2\beta}T+ \frac{\beta}{2}
\sum_{t=1}^T\|x^{t+1}-x^t\|^2+ \frac{(L_g+\kappa_{q})^2}{2\beta}
\sum_{t=1}^T\|\lambda^{t+1}\|^2 \\[10pt]\quad+ \frac{p\beta}{2}
\sum_{t=1}^T\|x^{t+1}-x^t\|^2+\alpha D_0\\[10pt]
=2 \kappa_f+ \frac{\kappa_{q}^2}{2\beta}T+ \frac{(1+p)\beta}{2}
\sum_{t=1}^T\|x^{t+1}-x^t\|^2+ \frac{(L_g+\kappa_{q})^2}{2\beta}
\sum_{t=1}^T\|\lambda^{t+1}\|^2+\alpha D_0,
\end{array}
$$
which proves (\ref{eq:kkt}).
 \end{proof}

Combining the results in Proposition \ref{prop:cregret}, Proposition \ref{prop:lambda} and Proposition \ref{prop:kkt}, we present the main theorem of this section.
\begin{theorem}\label{th:regrets}
Let Assumptions \ref{assu:lip}, \ref{assu:smooth}, \ref{assu:slater}, \ref{assu:psd}, \ref{assu:q}, \ref{assu:matrix}, \ref{assu:convex} be satisfied.  Then,  for $\sigma=T^{-1/4}$ and  $\alpha=T^{1/4}$,  the following assertions hold.
\begin{itemize}
\item[(i)] There exists a vector $w^{t+1} \in N_{\cC}(x^{t+1})$ such that the regret of Lagrangian residual is bounded by
\[
\left\|  \frac{1}{T}\sum_{t=1}^T \left[\nabla f_{t+1}(x^{t+1})+ \sum_{i=1}^p \lambda^{t+1}_i \nabla g_i(x^{t+1})+w^{t+1} \right]\right\|\\[10pt]
  \leq \varrho_0 T^{-1/8}+{\rm o}(T^{-1/8}),
\]
where
$$
\begin{array}{l}
\varrho_0=\frac{\kappa_{q}^2}{2}+2(1+p)\nu_g(\kappa_0+\kappa_1+\kappa_3)+\frac{(L_g+\kappa_{q})^2}{2}(\kappa_0+\kappa_1+\kappa_3)^2.
\end{array}
$$
\item[(ii)] The regret of constraint violation is
\[
 \frac{1}{T}\sum_{t=1}^T g_i(x^t)\leq (\nu_g(\kappa_0+\kappa_1+\kappa_3)+\kappa_g^2)  T^{-1/8}+{\rm o}(T^{-1/8}).
\]
\item[(iii)] The regret of complementarity residual is
\[
 \frac{1}{T}\sum_{t=1}^T\|\lambda^{t+1}-[\lambda^{t+1}+\sigma g(x^{t+1})]_+\| \leq \beta_0   T^{-1/4}+{\mathrm o}(T^{-1/4}).
\]
\end{itemize}
\end{theorem}
\begin{proof} It follows from Proposition \ref{prop:lambda} that
\[
\|\lambda^t\|\leq \psi(\sigma,\alpha,s)=\kappa_0+\kappa_1 \frac{\alpha}{s}+ \kappa_2
 \sigma+\kappa_3 \sigma s,
\]
where $\kappa_0,\kappa_1,\kappa_2,\kappa_3$ are defined by (\ref{eq:notations}).
For $\sigma=T^{-1/4}$ and $\alpha=T^{1/4}$, we take $s=T^{1/4}$ and hence
\begin{equation}\label{help12}
\|\lambda^t\|\leq\kappa_0+\kappa_1+\kappa_3+\kappa_2T^{-1/4}.
\end{equation}
Combining the results in Proposition \ref{prop:kkt}  and Lemma \ref{lem:3}, we have
\[
\begin{array}{ll}
 \left\| \frac{1}{T}\sum_{t=1}^T\cH_t\right\|&
\leq
\frac{2 \kappa_f}{T}+ \frac{\kappa_{q}^2}{2\beta}+ \frac{2(1+p)\beta}{\alpha T}
\left[  \frac{T}{\alpha}\kappa_f^2+\nu_g \sum_{t=1}^T\|\lambda^t\|+\frac{\sigma}{2}\nu_g^2T\right]\\[10pt]
&\quad\quad+ \frac{(L_g+\kappa_{q})^2}{2\beta T}
\sum_{t=1}^T\|\lambda^{t+1}\|^2+\frac{\alpha D_0}{T}.
\end{array}
\]
Taking $\beta=T^{1/8}$ and using (\ref{help12}), we obtain
\[
\begin{array}{ll}
 \left\| \frac{1}{T}\sum_{t=1}^T\cH_t\right\|\\[8pt]
\leq \frac{\kappa_{q}^2}{2}T^{-1/8}+2(1+p)T^{-1/8}[\kappa_f^2T^{-1/4}+\nu_g(\kappa_0+\kappa_1+\kappa_3+\kappa_2T^{-1/4})+\frac{\nu_g^2}{2}T^{-1/4}]\\[8pt]
\quad  +2 \kappa_fT^{-1}+\frac{(L_g+\kappa_{q})^2}{2}T^{-1/8}(\kappa_0+\kappa_1+\kappa_3+\kappa_2T^{-1/4})^2+T^{-3/4}D_0\\[8pt]
=\left[\frac{\kappa_{q}^2}{2}+2(1+p)\nu_g(\kappa_0+\kappa_1+\kappa_3)+\frac{(L_g+\kappa_{q})^2}{2}(\kappa_0+\kappa_1+\kappa_3)^2\right]T^{-1/8}+o(T^{-1/8}),
\end{array}
\]
which yields item (i).

From Proposition \ref{prop:cregret}, one has
\[
\frac{1}{T} \sum_{t=1}^T g_i(x^t)\leq   \frac{1}{\sigma T}\lambda^{T+1}_i+ \gamma\kappa_g^2+\frac{1}{\alpha T}\left[  \frac{1}{\gamma}+2\kappa_{q}\right] \left[  \frac{\kappa_f^2}{\alpha}T+\nu_g \sum_{t=1}^T\|\lambda^t\|+\frac{\sigma}{2}\nu_g^2T\right].
\]
Taking $\gamma=T^{-1/8}$ and using (\ref{help12}), we obtain
\[
\begin{array}{ll}
&\frac{1}{T} \sum_{t=1}^T g_i(x^t)\\[10pt]&\leq \left[T^{-1/8}+2\kappa_{q}T^{-1/4}\right] \left[\kappa_f^2T^{-1/4}+\nu_g(\kappa_0+\kappa_1+\kappa_3+\kappa_2T^{-1/4})+\frac{\nu_g^2}{2}T^{-1/4}\right] \\[10pt]
&\quad\quad+T^{-3/4}(\kappa_0+\kappa_1+\kappa_3+\kappa_2T^{-1/4})+ \kappa_g^2T^{-1/8},
\end{array}
\]
which proves item (ii).

Finally, we consider item (iii). First of all, we estimate $\|g(x^{t+1})- q^t(x^{t+1})\|$.  From Assumption  \ref{assu:smooth} and Assumption \ref{assu:matrix}, we have
\[
\begin{array}{ll}
\|g(x^{t+1})- q^t(x^{t+1})\|&=\left( \sum_{i=1}^p (g_i(x^{t+1})-q^t_i(x^{t+1}))^2\right)^{1/2}\\[10pt]
& \leq \left( \sum_{i=1}^p \left(
 \frac{L_g+\kappa_{q}}{2} \|x^{t+1}-x^t\|^2
 \right)^2\right)^{1/2}\\[10pt]
 &=\frac{\sqrt{p}(L_g+\kappa_{q})}{2}\|x^{t+1}-x^t\|^2.
\end{array}
\]
Then, from the definition of $\lambda^{t+1}$ and Lemma \ref{lem:1} we obtain
\[
\begin{array}{l}
\|\lambda^{t+1}-[\lambda^{t+1}+\sigma g(x^{t+1})]_+\|
= \|[\lambda^t +\sigma q^t(x^{t+1})]_+-[\lambda^{t+1}+\sigma g(x^{t+1})]_+\|\\[8pt]
\leq\|\lambda^{t+1}-\lambda^t +\sigma [g(x^{t+1})- q^t(x^{t+1})]\|\\[8pt]
\leq \beta_0\sigma +\frac{\sqrt{p}(L_g+\kappa_{q})\sigma}{2}\|x^{t+1}-x^t\|^2.
\end{array}
\]
Taking a summation and using Lemma  \ref{lem:3}, one has
\[
\begin{array}{ll}
\frac{1}{T}\sum_{t=1}^T\|\lambda^{t+1}-[\lambda^{t+1}+\sigma g(x^{t+1})]_+\|\\[10pt]\leq\beta_0\sigma +\frac{2\sqrt{p}(L_g+\kappa_{q})\sigma}{\alpha T}\left[ \frac{T}{\alpha}\kappa_f^2+\nu_g \sum_{t=1}^T\|\lambda^t\|+\frac{\sigma}{2}\nu_g^2T\right].
\end{array}
\]
Therefore, it follows from (\ref{help12}) that
\[
\begin{array}{ll}
\frac{1}{T}\sum_{t=1}^T\|\lambda^{t+1}-[\lambda^{t+1}+\sigma g(x^{t+1})]_+\|\\[8pt]
\leq 2\sqrt{p}(L_g+\kappa_{q})T^{-1/2}[\kappa_f^2T^{-1/4}+\nu_g(\kappa_0+\kappa_1+\kappa_3+\kappa_2T^{-1/4})+\frac{\nu_g^2}{2}T^{-1/4}]\\[10pt]
\quad\quad+\beta_0T^{-1/4},
\end{array}
\]
which completes the proof of item (iii).
\end{proof}

In the rest of this section, we analyze the objective reduction regret of the proposed algorithm under a setting where the objective function is a quadratic convex function, that is, $f_t(x)=q_0^t(x)$ and Assumption \ref{assu:psd} holds true. We emphasize that although the objective function is assumed to be convex, the feasible set $\Phi$ may still be non-convex.
\begin{proposition}\label{prop:qua-obj}
Let   Assumptions \ref{assu:lip}, \ref{assu:psd}, \ref{assu:q} be satisfied and $f_t(x)=q_0^t(x)$ for all $x\in\cC$. Let $\sigma=T^{-1/2}$ and $\alpha=T^{1/2}$. The following estimation holds:
\[
 \frac{1}{T}\sum_{t=1}^Tf_t(x^t)-\inf_{z \in \Phi}  \frac{1}{T}\sum_{t=1}^Tf_t(z) \leq
\left(\kappa_f^2+
 \frac{1}{2} \nu_g^2+\frac{1}{2}{\rm dist}^2\, (x^1, S^*)\right)T^{-1/2},
\]
where $S^*$ is the set of optimal solutions given by $S^*:=\argmin_{x\in\Phi}\sum_{t=1}^Tf_t(x)$.
\end{proposition}
\begin{proof} In view of (\ref{xna}), we have from Assumption \ref{assu:q} that, for all $z\in\cC$,
$$
\begin{array}{l}
q_0^t(x^{t+1})+ \frac{\alpha}{2} \|x^{t+1}-x^t\|^2 \\[8pt]
 \leq q_0^t(z)+ \frac{1}{2\sigma}\left[ \|[\lambda^t+\sigma q^t(z)]_+\|^2
-\|[\lambda^t+\sigma q^t(x^{t+1})]_+\|^2\right] \\[8pt]\quad\quad+ \frac{\alpha}{2}\left[\|z-x^t\|^2-\|z-x^{t+1}\|^2\right]\\[10pt]
\leq q_0^t(z)+ \frac{1}{2\sigma}\left[ \|[\lambda^t+\sigma g(z)]_+\|^2-\|\lambda^{t+1}\|^2\right] + \frac{\alpha}{2}\left[\|z-x^t\|^2-\|z-x^{t+1}\|^2\right].
\end{array}
$$
Rearranging terms and noticing $q_0^t(z)=f_t(z)$ we obtain
\begin{equation}\label{eq:aux1}
\begin{array}{ll}
f_t(x^t)+ \frac{\alpha}{4} \|x^{t+1}-x^t\|^2\\[8pt]
 \leq f_t(z)+\left( f_t(x^t)-q_0^t(x^{t+1}) - \frac{\alpha}{4} \|x^{t+1}-x^t\|^2\right)+ \frac{1}{2\sigma}(\|\lambda^t\|^2-\|\lambda^{t+1}\|^2) \\[10pt]
 \quad \quad +\langle \lambda^t, g(z)\rangle+\frac{\sigma}{2} \|g(z)\|^2+ \frac{\alpha}{2}\left[\|z-x^t\|^2-\|z-x^{t+1}\|^2\right].
\end{array}
\end{equation}
From Assumption \ref{assu:psd} and Assumption \ref{assu:lip} one has
\[
\begin{array}{ll}
f_t(x^t)-q_0^t(x^{t+1}) - \frac{\alpha}{4} \|x^{t+1}-x^t\|^2\\[8pt]
=\langle -\nabla f_t(x^t), x^{t+1}-x^t \rangle- \frac{1}{2}\langle \Theta^t_0( x^{t+1}-x^t), x^{t+1}-x^t \rangle- \frac{\alpha}{4} \|x^{t+1}-x^t\|^2\\[8pt]
\leq \frac{1}{\alpha}\|\nabla f_t(x^t)\|^2\leq \frac{1}{\alpha}\kappa_f^2.
\end{array}
\]
Therefore, from (\ref{eq:aux1}) and the fact that $\langle \lambda^t, g(z)\rangle\leq 0$ for all $z\in\Phi$ we have
\[
f_t(x^t)\leq f_t(z)+\frac{1}{\alpha}\kappa_f^2+ \frac{1}{2\sigma}(\|\lambda^t\|^2-\|\lambda^{t+1}\|^2)+\frac{\sigma}{2} \nu_g^2+ \frac{\alpha}{2}\left[\|z-x^t\|^2-\|z-x^{t+1}\|^2\right].
\]
Making a summation, we obtain
\[
\frac{1}{T}\sum_{t=1}^Tf_t(x^t)\leq\frac{1}{T}\sum_{t=1}^Tf_t(z)+\frac{\kappa_f^2}{\alpha}+\frac{\sigma\nu_g^2}{2}+\frac{\alpha}{2T}\|z-x^1\|^2.
\]
Hence, the claim is derived by noting that $\sigma=T^{-1/2}$ and $\alpha=T^{1/2}$.
\end{proof}
\section{OPMM for Online Optimization with Convex Constraints}\label{sec:cvx-cons}
In this section, we consider the online optimization problem with convex functional constraints, namely, the case that $g_1,\ldots, g_p$ are all convex functions. Moreover, we choose $\Theta^t_i=0$, $i=1,\ldots,p$ in OPMM.  In this case, Assumption \ref{assu:q} is naturally satisfied and Assumption \ref{assu:matrix} is reduced to the condition  $\|\Theta^t_0\|\leq \kappa_{q}$.
Further, under Assumption \ref{assu:psd}, the subproblem (\ref{xna}) is reduced to the following convex optimization problem
\begin{equation}\label{eq:subproblemt}
\min\limits_{x\in \cC} \left\{q_0^t(x)+ \frac{1}{2\sigma} \|[\lambda^t+\sigma g(x^t)+\sigma {\cal J}g(x^t) (x-x^t)]_+\|^2+\frac{\alpha}{2}\|x-x^t\|^2\right\}.
\end{equation}

For positively definite matrix $G \in {\mathbb S}^n$ and $x \in \bR^n$, we use ${\rm dist}^G_{\cC}(x)$ to denote
the weighted distance of $x$ from $\cC$, which is defined by
$$
{\rm dist}^G_{\cC}(x):=\inf_{u\in \cC} \|x-u\|_G,
$$
where $\|x\|_G=\sqrt{x^TGx}$ is the $G$-weighted norm of $x$. The $G$-weighted projection of $x$ onto $\cC$, denoted by
$\Pi^G_{\cC}(x)$, is defined by
$$
\Pi^G_{\cC}(x):=\argmin\limits_{u\in \cC}\|x-u\|_G.
$$
The following lemma is well-known, see, e.g., \cite[Example 3.31]{Beck2017}.
\begin{lemma}\label{eq:lemP}
For a closed convex set $\cC \subset \bR^n$ and a positively definite matrix $G \in {\mathbb S}^n$, let
$\pi (x):= \frac{1}{2}{\rm dist}^G_{\cC}(x)^2$.
Then, $\pi$ is continuously differentiable  and
\[
\nabla \pi (x)=G(x-\Pi^G_{\cC}(x)).
\]
\end{lemma}

By introducing artificial vectors $z$ and $w$, we can express Problem (\ref{eq:subproblemt}) as the following equivalent  convex quadratic programming problem
\begin{equation}\label{eq:subQ}
\begin{array}{ll}
\min\limits_{x,z,w} & q_0^t(x)+ \frac{\alpha}{2}\|x-x^t\|^2+ \frac{1}{2\sigma}\|z\|^2\\[10pt]
{\rm s.t.} & \lambda^t+\sigma g(x^t)+\sigma {\cal J} g(x^t) (x-x^t)-z +w =0,\\[10pt]
 & x\in \cC,\ z \geq 0,\ w\geq 0.
 \end{array}
\end{equation}
The Lagrangian function of Problem (\ref{eq:subQ}) is given by
\[
\begin{array}{ll}
L^t(x,z,w,y)\\[10pt]= q_0^t(x)+ \frac{\alpha}{2}\|x-x^t\|^2+ \frac{1}{2\sigma}\|z\|^2 + \langle y, \lambda^t+\sigma g(x^t)+\sigma {\cal J}g(x^t) (x-x^t)-z +w \rangle.
\end{array}
\]
Then, the dual of Problem (\ref{eq:subQ}) is expressed as
\[
\begin{array}{l}
\max\limits_{y\in\bR^p}\inf\limits_{x\in \cC}\inf\limits_{z \geq 0}\inf\limits_{w\geq 0} L^t(x,z,w,y)\\[8pt]
=\max\limits_{y\in\bR^p}\left\{\langle y, \lambda^t+\sigma g(x^t)\rangle+\inf\limits_{x \in \cC}[q_0^t(x)+ \frac{\alpha}{2}\|x-x^t\|^2+\langle y, \sigma {\cal J}g(x^t) (x-x^t)\rangle]
\right.\\[10pt]
\quad\quad\quad\quad+\inf\limits_{w\geq 0}y^Tw+\left.\inf\limits_{z \geq 0}\left[ \frac{1}{2\sigma}\|z\|^2 - \langle y, z  \rangle\right] \right\}\\[10pt]
=\max\limits_{y\geq 0}\left\{\langle y, \lambda^t+\sigma g(x^t)\rangle- \frac{\sigma}{2} \|y\|^2 +f_t(x^t)- \frac{1}{2} \| \nabla f_t(x^t)+\sigma {\cal J}g(x^t)^Ty\|^2_{[H^t]^{-1}}
\right.\\[10pt]
\quad\quad\quad\quad+\left.\inf_{x \in \cC}\left[
   \frac{1}{2}\left \|x-\Big(x^t-[H^t]^{-1}(\nabla f_t(x^t)+\sigma {\cal J}g(x^t)^Ty)\Big)\right\|^2_{H^t}
  \right] \right\}\\[10pt]

  =\max\limits_{y\geq 0}\left\{\langle y, \lambda^t+\sigma g(x^t)\rangle- \frac{\sigma}{2} \|y\|^2 +f_t(x^t)- \frac{1}{2} \| \nabla f_t(x^t)+\sigma {\cal J}g(x^t)^Ty\|^2_{[H^t]^{-1}}
  \right.\\[10pt]
\quad \quad +\left.\left[
   \frac{1}{2} {\rm dist}_{\cC}^{H^t}\Big(x^t-[H^t]^{-1}(\nabla f_t(x^t)+\sigma {\cal J}g(x^t)^Ty)\Big)^2
  \right]\right\},
\end{array}
\]
where $H^t:=\Theta^t_0+\alpha I$.
Therefore, when we derive the optimal solution $y^t$ by solving the dual problem,  from the duality theory, the solution of the subproblem (\ref{xna}) is given by
\[
x^{t+1}=\Pi_{\cC}^{H^t}\Big(x^t-[H^t]^{-1}(\nabla f_t(x^t)+\sigma {\cal J}g(x^t)^Ty^t)\Big).
\]

Based on the above analysis, OPMM for the online problem with convex constraints  can be rewritten as
in Algorithm \ref{algo:P-OPMM}.

\begin{algorithm2e}[!htp]\caption{
Projection version of  OPMM for online non-convex optimization with convex constraints.}\label{algo:P-OPMM}
  \SetKwInOut{Input}{Input}
  \Input{$\lambda^1=0$, $x^1 \in \cC$, $\sigma >0$ and $\alpha>0$, receive a cost function $f_1(\cdot)$. }
  \BlankLine
  \For{$t\leftarrow 1$ \KwTo $T$}{
  \emph{Choose $\Theta^t_0 \in \mathbb S^n$ and set $H^t:=\Theta^t_0+\alpha I$. Solve the following convex optimization problem to obtain $y^t$:}\
     \begin{equation}\label{eq:dual-sub}
    \begin{array}{lr}
   \max\limits_{y\geq 0}\left\{- \frac{\sigma}{2} \|y\|^2+\langle y, \lambda^t+\sigma g(x^t)\rangle- \frac{1}{2} \| \nabla f_t(x^t)+\sigma {\cal J}g(x^t)^Ty\|^2_{[H^t]^{-1}}\right.\\[10pt]
 \quad \quad \quad \quad \quad +\left.\left[
   \frac{1}{2} {\rm dist}_{\cC}^{H^t}\Big(x^t-[H^t]^{-1}(\nabla f_t(x^t)+\sigma {\cal J}g(x^t)^Ty)\Big)^2
  \right]\right\}.
\end{array}
\end{equation}

  \emph{Compute}\
 \begin{equation}\label{xnad}
 x^{t+1}=\Pi_{\cC}^{H^t}\Big(x^t-[H^t]^{-1}(\nabla f_t(x^t)+\sigma {\cal J}g(x^t)^Ty^t)\Big).
\end{equation}
	
  \emph{Update}\
\[
\lambda^{t+1}_i=\left[
\lambda^t_i+\sigma (g_i(x^t)+\langle \nabla g_i(x^t), x^{t+1}-x^t \rangle
)\right]_+,\,\, i=1,\ldots,p.
\]

  \emph{Receive a cost function $f_{t+1}(\cdot)$.}\
  }
\end{algorithm2e}

We now discuss the relationship between $y^t$ and $\lambda^{t+1}$.
\begin{proposition}\label{prop-rely}
Let $\omega_t (y)$ denote the objective function of Problem (\ref{eq:dual-sub}). Then,
\begin{equation}\label{eq:relylam}
\lambda^{t+1}=[\nabla \omega_t (y^t)+\sigma y^t]_+.
\end{equation}
\end{proposition}
\begin{proof} It follows from Lemma \ref{eq:lemP} that
$$
\begin{array}{ll}
&\nabla \omega_t (y)\\[10pt] &= - \sigma y+\lambda^t+\sigma g(x^t)-\sigma {\cal J}g(x^t)\left[
[H^t]^{-1}(\nabla f_t(x^t)+\sigma {\cal J} g(x^t)^Ty)\right]\\[10pt]
&\quad \quad  -\sigma {\cal J}g(x^t)[H^t]^{-1}H^t\Big[x^t-[H^t]^{-1}(\nabla f_t(x^t)+\sigma {\cal J}g(x^t)^Ty)\\[10pt]
& \quad \quad -\Pi_{\cC}^{H^t}\Big(x^t-[H^t]^{-1}(\nabla f_t(x^t)+\sigma {\cal J}g(x^t)^Ty)\Big)\Big]\\[10pt]
&= - \sigma y+\lambda^t+\sigma g(x^t)  -\sigma {\cal J}g(x^t)\Big[x^t -\Pi_{\cC}^{H^t}\Big(x^t-[H^t]^{-1}(\nabla f_t(x^t)+\sigma {\cal J}g(x^t)^Ty)\Big)\Big].
\end{array}
$$
Hence, from (\ref{xnad}) we have
$$
\nabla \omega_t (y^t)=-\sigma y^t+\lambda^t+\sigma \Big[g(x^t)+{\cal J}g(x^t)(x^{t+1}-x^t)\Big],
$$
which implies
$$
\lambda^t+\sigma \Big[g(x^t)+{\cal J}g(x^t)(x^{t+1}-x^t)\Big]=\nabla \omega_t (y^t)+\sigma y^t.
$$
The claim is derived by noticing the definition of $\lambda^{t+1}$.
\end{proof}

Furthermore, if we choose a  scalar $\eta_t>0$ such that  $\Theta^t_0=\eta_t I$ satisfies the required assumptions, we obtain that $H^t=(\alpha +\eta_t)I$, $[H^t]^{-1}=(\alpha +\eta_t)^{-1}I$ and  Problem (\ref{eq:dual-sub}) is equivalent to
\begin{equation}\label{eq:dual-subI}
    \begin{array}{lr}
   \max\limits_{y\geq 0}\left\{- \frac{\sigma}{2} \|y\|^2+\langle y, \lambda^t+\sigma g(x^t)\rangle- \frac{1}{2(\alpha+\eta_t)} \| \nabla f_t(x^t)+\sigma {\cal J}g(x^t)^Ty\|^2\right.\\[10pt]
\quad \quad \quad \quad \quad +\left.\left[
   \frac{\alpha +\eta_t}{2} {\rm dist}\Big(x^t-[\alpha +\eta_t]^{-1}(\nabla f_t(x^t)+\sigma {\cal J}g(x^t)^Ty),\cC\Big)^2
  \right]\right\}.
\end{array}
    \end{equation}
Hence, the formula  (\ref{xnad}) is reduced to
\[
     x^{t+1}=\Pi_{\cC}\Big(x^t-[\alpha+\eta_t]^{-1}(\nabla f_t(x^t)+\sigma {\cal J}g(x^t)^Ty^t)\Big),
\]
where $y^t$ is the solution to Problem (\ref{eq:dual-subI}). In this case, at each iteration of Algorithm \ref{algo:P-OPMM} the main calculations are computing a projection and solving a relatively simple convex minimization problem which make the algorithm easy to be implemented.

\section{Conclusion}\label{sec:conclusion}
In this paper, we present a proximal method of multipliers with quadratic approximations (OPMM) for solving an online non-convex  optimization  with (possibly non-convex)  inequality  constraints. We  show that,  this algorithm exhibits  ${\cO}(T^{-1/8})$ Lagrangian residual  regret, ${\cO}(T^{-1/8})$ regret of constraint violation  and ${\cO}(T^{-1/4})$  complementarity residual regret if parameters in the algorithm are properly chosen, where $T$ denotes the number of iterations.
We also show that, for  the  case when  the constraint functions are all convex, the projection version of OPMM provides a practical way for finding a decision sequence  $\{x^1,x^2,\ldots, x^{T}\}$. To the best of our knowledge, the regret analysis of numerical methods for online non-convex optimization with long term constraints  has not been studied in the literature yet.

We note that, even for the simple bounded box set $\cC=[a,b] \subset \bR^n$,  the analysis of regrets requires the exact solution to the subproblem  (\ref{eq:dual-subI}). How to obtain the regret bounds when  the subproblem is inexactly solved is an important future research topic worth considering.
%

\backmatter
\bmhead{Acknowledgments}
The authors would like to thank  the two reviewers for the valuable suggestions. The research is supported
by  the National Natural Science Foundation of China (No. 11731013, No. 11971089  and No. 11871135).

\begin{appendices}

\section{Proof of Lemma \ref{lem:l7}}\label{app:lem}
\begin{proof}
Let
$$
r= \frac{\zeta}{4t_0\delta_{\max}^2}, \quad \rho=1- \frac{\zeta^2}{8\delta_{\max}^2},
$$
then it yields that
$
\rho=1- \frac{rt_0}{2}\zeta.
$
Define $\eta (t)=Z_{t+t_0}-Z_t$, then from $\lvert Z_{t+1}-Z_t\rvert\leq \delta_{\max}$ we have
$
\lvert\eta (t)\rvert\leq t_0 \delta_{\max}
$
and hence
\begin{equation}\label{eq:ineq2}
\lvert r \eta(t)\rvert \leq  \frac{\zeta}{4t_0\delta_{\max}^2} \cdot t_0 \delta_{\max} = \frac{\zeta}{4
\delta_{\max}} \leq 1.
\end{equation}
From (\ref{eq:ineq2}) and the following inequality
$$
e^{\tau} \leq 1+\tau+2 \tau^2 \mbox{ when }  \lvert\tau\rvert <1,
$$
we obtain
\[
\begin{array}{ll}
e^{rZ_{t+t_0}} & = e^{rZ_t}e^{r\eta(t)}\\[10pt]
& \leq e^{rZ_t} [1+r \eta (t)+2r^2 t_0^2\delta_{\max}^2]\\[10pt]
&= e^{rZ_t} [1+r \eta (t)+r t_0\zeta/2].
\end{array}
\]
Case 1: $Z_t\geq \theta$. In this case, one has from (\ref{eq:Y9}) that $\eta(t) \leq -t_0 \zeta$ and hence
\begin{equation}\label{eq:ineq4}
\begin{array}{ll}
e^{rZ_{t+t_0}}
& \leq e^{rZ_t} [1-rt_0 \zeta+r t_0\zeta/2]\\[8pt]
& = e^{rZ_t} [1-rt_0 \zeta/2]\\[8pt]
&=\rho e^{rZ_t}.
\end{array}
\end{equation}
Case 2: $Z_t< \theta$. In this case, one has  $\eta(t) \leq t_0 \delta_{\max}$ and hence
\begin{equation}\label{eq:ineq5}
\begin{array}{ll}
e^{rZ_{t+t_0}} & = e^{rZ_t}e^{r \eta (t)}\\[8pt]
&\leq e^{rZ_t}e^{r t_0 \delta_{\max}}\\[8pt]
&\leq e^{r\theta}e^{r t_0 \delta_{\max}}.
\end{array}
\end{equation}
Combining (\ref{eq:ineq4}) and (\ref{eq:ineq5}), we obtain
\begin{equation}\label{eq:ineq6}
e^{rZ_{t+t_0}} \leq \rho e^{rZ_t}+e^{r\theta}e^{r t_0 \delta_{\max}}.
\end{equation}

 We next prove the following inequality by induction,
\begin{equation}\label{eq:ineq7}
e^{rZ_t} \leq  \frac{1}{1-\rho}e^{r\theta}e^{r t_0 \delta_{\max}},\ t\in \{0,1,\ldots\}.
\end{equation}
We first consider the case $t \in \{0,1,\ldots, t_0\}$. From $\lvert Z_{t+1}-Z_t\rvert  \leq \delta_{\max}$ and $Z_0=0$ we have  $Z_t \leq t \delta_{\max}$. This, together with the fact that $\frac{e^{r\theta}}{1-\rho}\geq 1$, implies
$$
e^{rZ_t} \leq e^{rt\delta_{\max}}\leq e^{r t_0\delta_{\max}} \leq e^{r t_0\delta_{\max}} \frac{e^{r\theta}}{1-\rho}.
$$
Hence, (\ref{eq:ineq7}) is satisfied for all $t \in \{0,1,\ldots, t_0\}$.
We now assume that (\ref{eq:ineq7}) holds true
for all $t \in \{t_0+1,\ldots, \tau\}$ with arbitrary $\tau > t_0$. Consider $t=\tau+1$. By (\ref{eq:ineq6}), we have
$$
\begin{array}{ll}
e^{rZ_{\tau+1}}\leq \rho e^{rZ_{\tau+1-t_0}}+e^{r\theta}e^{r t_0 \delta_{\max}} \\[8pt]  \leq \rho  e^{r t_0\delta_{\max}} \frac{e^{r\theta}}{1-\rho}+e^{r\theta}e^{r t_0 \delta_{\max}} = e^{r t_0\delta_{\max}} \frac{e^{r\theta}}{1-\rho}.
\end{array}
$$
Therefore, the inequality (\ref{eq:ineq7}) holds for all $t\in \{0,1,\ldots\}$. Taking logarithm on both sides of (\ref{eq:ineq7})  and dividing by $r$ yields
$$
Z_t  \leq \theta +t_0 \delta_{\max} + \frac{1}{r} \log \left ( \frac{1}{1-\rho} \right)= \theta +t_0 \delta_{\max} + t_0 \frac{4 \delta_{\max}^2}{\zeta} \log \left ( \frac{8\delta_{\max}^2}{\zeta^2} \right).
$$
The proof is completed.
 \end{proof}




\end{appendices}

\bibliographystyle{sn-mathphys}
\bibliography{ref}


\begin{thebibliography}{30}
\ifx \bisbn   \undefined \def \bisbn  #1{ISBN #1}\fi
\ifx \binits  \undefined \def \binits#1{#1}\fi
\ifx \bauthor  \undefined \def \bauthor#1{#1}\fi
\ifx \batitle  \undefined \def \batitle#1{#1}\fi
\ifx \bjtitle  \undefined \def \bjtitle#1{#1}\fi
\ifx \bvolume  \undefined \def \bvolume#1{\textbf{#1}}\fi
\ifx \byear  \undefined \def \byear#1{#1}\fi
\ifx \bissue  \undefined \def \bissue#1{#1}\fi
\ifx \bfpage  \undefined \def \bfpage#1{#1}\fi
\ifx \blpage  \undefined \def \blpage #1{#1}\fi
\ifx \burl  \undefined \def \burl#1{\textsf{#1}}\fi
\ifx \doiurl  \undefined \def \doiurl#1{\url{https://doi.org/#1}}\fi
\ifx \betal  \undefined \def \betal{\textit{et al.}}\fi
\ifx \binstitute  \undefined \def \binstitute#1{#1}\fi
\ifx \binstitutionaled  \undefined \def \binstitutionaled#1{#1}\fi
\ifx \bctitle  \undefined \def \bctitle#1{#1}\fi
\ifx \beditor  \undefined \def \beditor#1{#1}\fi
\ifx \bpublisher  \undefined \def \bpublisher#1{#1}\fi
\ifx \bbtitle  \undefined \def \bbtitle#1{#1}\fi
\ifx \bedition  \undefined \def \bedition#1{#1}\fi
\ifx \bseriesno  \undefined \def \bseriesno#1{#1}\fi
\ifx \blocation  \undefined \def \blocation#1{#1}\fi
\ifx \bsertitle  \undefined \def \bsertitle#1{#1}\fi
\ifx \bsnm \undefined \def \bsnm#1{#1}\fi
\ifx \bsuffix \undefined \def \bsuffix#1{#1}\fi
\ifx \bparticle \undefined \def \bparticle#1{#1}\fi
\ifx \barticle \undefined \def \barticle#1{#1}\fi
\bibcommenthead
\ifx \bconfdate \undefined \def \bconfdate #1{#1}\fi
\ifx \botherref \undefined \def \botherref #1{#1}\fi
\ifx \url \undefined \def \url#1{\textsf{#1}}\fi
\ifx \bchapter \undefined \def \bchapter#1{#1}\fi
\ifx \bbook \undefined \def \bbook#1{#1}\fi
\ifx \bcomment \undefined \def \bcomment#1{#1}\fi
\ifx \oauthor \undefined \def \oauthor#1{#1}\fi
\ifx \citeauthoryear \undefined \def \citeauthoryear#1{#1}\fi
\ifx \endbibitem  \undefined \def \endbibitem {}\fi
\ifx \bconflocation  \undefined \def \bconflocation#1{#1}\fi
\ifx \arxivurl  \undefined \def \arxivurl#1{\textsf{#1}}\fi
\csname PreBibitemsHook\endcsname

\bibitem{MSF2017}
\begin{botherref}
\oauthor{\bsnm{M{\'{a}}rquez{-}Neila}, \binits{P.}},
\oauthor{\bsnm{Salzmann}, \binits{M.}},
\oauthor{\bsnm{Fua}, \binits{P.}}:
Imposing Hard Constraints on Deep Networks: Promises and Limitations
(2017).
\url{http://arxiv.org/abs/1706.02025}
\end{botherref}
\endbibitem

\bibitem{NPM2019}
\begin{bchapter}
\bauthor{\bsnm{Nandwani}, \binits{Y.}},
\bauthor{\bsnm{Pathak}, \binits{A.}},
\bauthor{\bsnm{Mausam}},
\bauthor{\bsnm{Singla}, \binits{P.}}:
\bctitle{A primal dual formulation for deep learning with constraints}.
In: \bbtitle{Advances in Neural Information Processing Systems 32},
pp. \bfpage{12157}--\blpage{12168}
(\byear{2019})
\end{bchapter}
\endbibitem

\bibitem{CJGWNYS2019}
\begin{barticle}
\bauthor{\bsnm{Cotter}, \binits{A.}},
\bauthor{\bsnm{Jiang}, \binits{H.}},
\bauthor{\bsnm{Gupta}, \binits{M.}},
\bauthor{\bsnm{Wang}, \binits{S.}},
\bauthor{\bsnm{Narayan}, \binits{T.}},
\bauthor{\bsnm{You}, \binits{S.}},
\bauthor{\bsnm{Sridharan}, \binits{K.}}:
\batitle{Optimization with non-differentiable constraints with applications to
  fairness, recall, churn, and other goals}.
\bjtitle{Journal of Machine Learning Research}
\bvolume{20}(\bissue{172}),
\bfpage{1}--\blpage{59}
(\byear{2019})
\end{barticle}
\endbibitem

\bibitem{SZSEEGF2014}
\begin{bchapter}
\bauthor{\bsnm{Szegedy}, \binits{C.}},
\bauthor{\bsnm{Zaremba}, \binits{W.}},
\bauthor{\bsnm{Sutskever}, \binits{I.}},
\bauthor{\bsnm{{Bruna Estrach}}, \binits{J.}},
\bauthor{\bsnm{Erhan}, \binits{D.}},
\bauthor{\bsnm{Goodfellow}, \binits{I.}},
\bauthor{\bsnm{Fergus}, \binits{R.}}:
\bctitle{Intriguing Properties of Neural Networks}.
(\byear{2014}).
\bcomment{2nd International Conference on Learning Representations (ICLR)}
\end{bchapter}
\endbibitem

\bibitem{Goodfellow2014}
\begin{bchapter}
\bauthor{\bsnm{Goodfellow}, \binits{I.}},
\bauthor{\bsnm{Pouget-Abadie}, \binits{J.}},
\bauthor{\bsnm{Mirza}, \binits{M.}},
\bauthor{\bsnm{Xu}, \binits{B.}},
\bauthor{\bsnm{Warde-Farley}, \binits{D.}},
\bauthor{\bsnm{Ozair}, \binits{S.}},
\bauthor{\bsnm{Courville}, \binits{A.}},
\bauthor{\bsnm{Bengio}, \binits{Y.}}:
\bctitle{Generative adversarial nets}.
In: \beditor{\bsnm{Ghahramani}, \binits{Z.}},
\beditor{\bsnm{Welling}, \binits{M.}},
\beditor{\bsnm{Cortes}, \binits{C.}},
\beditor{\bsnm{Lawrence}, \binits{N.D.}},
\beditor{\bsnm{Weinberger}, \binits{K.Q.}} (eds.)
\bbtitle{Advances in Neural Information Processing Systems 27},
pp. \bfpage{2672}--\blpage{2680}
(\byear{2014})
\end{bchapter}
\endbibitem

\bibitem{Kalai2005}
\begin{barticle}
\bauthor{\bsnm{Kalai}, \binits{A.}},
\bauthor{\bsnm{Vempala}, \binits{S.}}:
\batitle{Efficient algorithms for online decision problems}.
\bjtitle{J. Comput. System Sci.}
\bvolume{71}(\bissue{3}),
\bfpage{291}--\blpage{307}
(\byear{2005})
\end{barticle}
\endbibitem

\bibitem{Shai2007a}
\begin{botherref}
\oauthor{\bsnm{Shalev-Shwartz}, \binits{S.}}:
Online learning: Theory, algorithms, and applications.
PhD thesis,
The Hebrew University
(2007)
\end{botherref}
\endbibitem

\bibitem{Shai2007b}
\begin{barticle}
\bauthor{\bsnm{Shalev-Shwartz}, \binits{S.}},
\bauthor{\bsnm{Singer}, \binits{Y.}}:
\batitle{A primal-dual perspective of online learning algorithms}.
\bjtitle{Machine Learning}
\bvolume{69}(\bissue{2}),
\bfpage{115}--\blpage{142}
(\byear{2007})
\end{barticle}
\endbibitem

\bibitem{Kiv1997}
\begin{barticle}
\bauthor{\bsnm{Kivinen}, \binits{J.}},
\bauthor{\bsnm{Warmuth}, \binits{M.K.}}:
\batitle{Exponentiated gradient versus gradient descent for linear predictors}.
\bjtitle{Inform. and Comput.}
\bvolume{132}(\bissue{1}),
\bfpage{1}--\blpage{63}
(\byear{1997})
\end{barticle}
\endbibitem

\bibitem{Rosenblatt1958}
\begin{barticle}
\bauthor{\bsnm{Rosenblatt}, \binits{F.}}:
\batitle{The perceptron: A probabilistic model for information storage and
  organization in the brain.}
\bjtitle{Psychological Review}
\bvolume{65}(\bissue{6}),
\bfpage{386}--\blpage{408}
(\byear{1958})
\end{barticle}
\endbibitem

\bibitem{Littlestone1988}
\begin{barticle}
\bauthor{\bsnm{Littlestone}, \binits{N.}}:
\batitle{Learning quickly when irrelevant attributes abound: A new
  linear-threshold algorithm}.
\bjtitle{Machine Learning}
\bvolume{2}(\bissue{4}),
\bfpage{285}--\blpage{318}
(\byear{1988})
\end{barticle}
\endbibitem

\bibitem{MRTal2012}
\begin{bbook}
\bauthor{\bsnm{Mohri}, \binits{M.}},
\bauthor{\bsnm{Rostamizadeh}, \binits{A.}},
\bauthor{\bsnm{Talwalkar}, \binits{A.}}:
\bbtitle{Foundations of Machine Learning}.
\bsertitle{Adaptive Computation and Machine Learning}.
\bpublisher{MIT Press},
\blocation{Cambridge, MA}
(\byear{2012})
\end{bbook}
\endbibitem

\bibitem{SSS2014}
\begin{bbook}
\bauthor{\bsnm{Shalev-Shwartz}, \binits{S.}},
\bauthor{\bsnm{Ben-David}, \binits{S.}}:
\bbtitle{Understanding Machine Learning: From Theory to Algorithms}.
\bpublisher{Cambridge University Press},
\blocation{New York, NY}
(\byear{2014})
\end{bbook}
\endbibitem

\bibitem{Shai2011}
\begin{barticle}
\bauthor{\bsnm{Shalev-Shwartz}, \binits{S.}}:
\batitle{Online learning and online convex optimization}.
\bjtitle{Foundations and Trends{\textregistered} in Machine Learning}
\bvolume{4}(\bissue{2}),
\bfpage{107}--\blpage{194}
(\byear{2011})
\end{barticle}
\endbibitem

\bibitem{Hazan2015}
\begin{barticle}
\bauthor{\bsnm{Hazan}, \binits{E.}}:
\batitle{Introduction to online convex optimization}.
\bjtitle{Foundations and Trends{\textregistered} in Optimization}
\bvolume{2}(\bissue{3-4}),
\bfpage{157}--\blpage{325}
(\byear{2015})
\end{barticle}
\endbibitem

\bibitem{EBG2011}
\begin{barticle}
\bauthor{\bsnm{{Ertekin}}, \binits{S.}},
\bauthor{\bsnm{{Bottou}}, \binits{L.}},
\bauthor{\bsnm{{Giles}}, \binits{C.L.}}:
\batitle{Nonconvex online support vector machines}.
\bjtitle{IEEE Transactions on Pattern Analysis and Machine Intelligence}
\bvolume{33}(\bissue{2}),
\bfpage{368}--\blpage{381}
(\byear{2011})
\end{barticle}
\endbibitem

\bibitem{GPSB2011}
\begin{botherref}
\oauthor{\bsnm{Gasso}, \binits{G.}},
\oauthor{\bsnm{Pappaioannou}, \binits{A.}},
\oauthor{\bsnm{Spivak}, \binits{M.}},
\oauthor{\bsnm{Bottou}, \binits{L.}}:
Batch and online learning algorithms for nonconvex neyman-pearson
  classification.
ACM Trans. Intell. Syst. Technol.
\textbf{2}(3)
(2011)
\end{botherref}
\endbibitem

\bibitem{GaoZZhang2018}
\begin{bchapter}
\bauthor{\bsnm{Gao}, \binits{X.}},
\bauthor{\bsnm{Li}, \binits{X.}},
\bauthor{\bsnm{Zhang}, \binits{S.}}:
\bctitle{Online learning with non-convex losses and non-stationary regret}.
In: \beditor{\bsnm{Storkey}, \binits{A.}},
\beditor{\bsnm{Perez-Cruz}, \binits{F.}} (eds.)
\bbtitle{Proceedings of the Twenty-First International Conference on Artificial
  Intelligence and Statistics}.
\bsertitle{Proceedings of Machine Learning Research},
vol. \bseriesno{84},
pp. \bfpage{235}--\blpage{243}.
\bconflocation{Playa Blanca, Lanzarote, Canary Islands}
(\byear{2018})
\end{bchapter}
\endbibitem

\bibitem{HSZ2017}
\begin{bchapter}
\bauthor{\bsnm{Hazan}, \binits{E.}},
\bauthor{\bsnm{Singh}, \binits{K.}},
\bauthor{\bsnm{Zhang}, \binits{C.}}:
\bctitle{Efficient regret minimization in non-convex games}.
In: \beditor{\bsnm{Precup}, \binits{D.}},
\beditor{\bsnm{Teh}, \binits{Y.W.}} (eds.)
\bbtitle{Proceedings of the 34th International Conference on Machine Learning}.
\bsertitle{Proceedings of Machine Learning Research},
vol. \bseriesno{70},
pp. \bfpage{1433}--\blpage{1441}.
\bconflocation{International Convention Centre, Sydney, Australia}
(\byear{2017})
\end{bchapter}
\endbibitem

\bibitem{LH2020}
\begin{barticle}
\bauthor{\bsnm{Le~Thi}, \binits{H.A.}},
\bauthor{\bsnm{Ho}, \binits{V.T.}}:
\batitle{Online learning based on online dca and application to online
  classification}.
\bjtitle{Neural Computation}
\bvolume{32}(\bissue{4}),
\bfpage{759}--\blpage{793}
(\byear{2020})
\end{barticle}
\endbibitem

\bibitem{YDHTW2018}
\begin{bchapter}
\bauthor{\bsnm{Yang}, \binits{L.}},
\bauthor{\bsnm{Deng}, \binits{L.}},
\bauthor{\bsnm{Hajiesmaili}, \binits{M.H.}},
\bauthor{\bsnm{Tan}, \binits{C.}},
\bauthor{\bsnm{Wong}, \binits{W.S.}}:
\bctitle{An optimal algorithm for online non-convex learning}.
In: \bbtitle{Abstracts of the 2018 ACM International Conference on Measurement
  and Modeling of Computer Systems}.
\bsertitle{SIGMETRICS ’18},
pp. \bfpage{41}--\blpage{43},
\bconflocation{New York, NY, USA}
(\byear{2018})
\end{bchapter}
\endbibitem

\bibitem{AgarwalHazn2019}
\begin{bchapter}
\bauthor{\bsnm{Agarwal}, \binits{N.}},
\bauthor{\bsnm{Gonen}, \binits{A.}},
\bauthor{\bsnm{Hazan}, \binits{E.}}:
\bctitle{Learning in non-convex games with an optimization oracle}.
In: \beditor{\bsnm{Beygelzimer}, \binits{A.}},
\beditor{\bsnm{Hsu}, \binits{D.}} (eds.)
\bbtitle{Proceedings of the Thirty-Second Conference on Learning Theory}.
\bsertitle{Proceedings of Machine Learning Research},
vol. \bseriesno{99},
pp. \bfpage{18}--\blpage{29}.
\bconflocation{Phoenix, USA}
(\byear{2019})
\end{bchapter}
\endbibitem

\bibitem{SN2019}
\begin{botherref}
\oauthor{\bsnm{Suggala}, \binits{A.S.}},
\oauthor{\bsnm{Netrapalli}, \binits{P.}}:
Online Non-Convex Learning: Following the Perturbed Leader is Optimal.
https://arxiv.org/abs/1903.08110
(2019)
\end{botherref}
\endbibitem

\bibitem{RBGM2019}
\begin{botherref}
\oauthor{\bsnm{Roy}, \binits{A.}},
\oauthor{\bsnm{Balasubramanian}, \binits{K.}},
\oauthor{\bsnm{Ghadimi}, \binits{S.}},
\oauthor{\bsnm{Mohapatra}, \binits{P.}}:
Multi-Point Bandit Algorithms for Nonstationary Online Nonconvex Optimization.
https://arxiv.org/abs/1907.13616
(2019)
\end{botherref}
\endbibitem

\bibitem{MJY2012}
\begin{barticle}
\bauthor{\bsnm{Mahdavi}, \binits{M.}},
\bauthor{\bsnm{Jin}, \binits{R.}},
\bauthor{\bsnm{Yang}, \binits{T.}}:
\batitle{Trading regret for efficiency: online convex optimization with long
  term constraints}.
\bjtitle{J. Mach. Learn. Res.}
\bvolume{13},
\bfpage{2503}--\blpage{2528}
(\byear{2012})
\end{barticle}
\endbibitem

\bibitem{JHA2016}
\begin{bchapter}
\bauthor{\bsnm{Jenatton}, \binits{R.}},
\bauthor{\bsnm{Huang}, \binits{J.}},
\bauthor{\bsnm{Archambeau}, \binits{C.}}:
\bctitle{Adaptive algorithms for online convex optimization with long-term
  constraints}.
In: \bbtitle{Proceedings of The 33rd International Conference on Machine
  Learning}.
\bsertitle{Proceedings of Machine Learning Research},
vol. \bseriesno{48},
pp. \bfpage{402}--\blpage{411}.
\bconflocation{New York, New York, USA}
(\byear{2016})
\end{bchapter}
\endbibitem

\bibitem{YNeely2016}
\begin{barticle}
\bauthor{\bsnm{Yu}, \binits{H.}},
\bauthor{\bsnm{Neely}, \binits{M.J.}}:
\batitle{A low complexity algorithm with {$O(\sqrt T)$} regret and {$O(1)$}
  constraint violations for online convex optimization with long term
  constraints}.
\bjtitle{J. Mach. Learn. Res.}
\bvolume{21},
\bfpage{1}--\blpage{24}
(\byear{2020})
\end{barticle}
\endbibitem

\bibitem{Rockafellar76A}
\begin{barticle}
\bauthor{\bsnm{Rockafellar}, \binits{R.T.}}:
\batitle{Augmented {L}agrangians and applications of the proximal point
  algorithm in convex programming}.
\bjtitle{Math. Oper. Res.}
\bvolume{1}(\bissue{2}),
\bfpage{97}--\blpage{116}
(\byear{1976})
\end{barticle}
\endbibitem

\bibitem{YNeely2017}
\begin{bchapter}
\bauthor{\bsnm{Yu}, \binits{H.}},
\bauthor{\bsnm{Neely}, \binits{M.J.}},
\bauthor{\bsnm{Wei}, \binits{X.}}:
\bctitle{Online Convex Optimization with Stochastic Constraints}.
In: \bbtitle{Advances in Neural Information Processing Systems},
pp. \bfpage{1428}--\blpage{1438}
(\byear{2017})
\end{bchapter}
\endbibitem

\bibitem{Beck2017}
\begin{bbook}
\bauthor{\bsnm{Beck}, \binits{A.}}:
\bbtitle{First-order Methods in Optimization}.
\bsertitle{MOS-SIAM Series on Optimization},
vol. \bseriesno{25}.
\bpublisher{Society for Industrial and Applied Mathematics (SIAM)},
\blocation{Philadelphia, PA}
(\byear{2017})
\end{bbook}
\endbibitem

\end{thebibliography}


\end{document}